\documentclass[11pt]{amsart}
\usepackage{amsmath, amscd, amssymb}
\usepackage{latexsym, amssymb, amsmath, amscd, amsfonts}
\usepackage[mathscr]{eucal}
\usepackage{mathrsfs}
\usepackage{concrete,euler,textcomp}
\usepackage{ytableau}
\usepackage{graphicx}
\usepackage[T1]{fontenc}
\usepackage{url}
\usepackage[all]{xy}
\usepackage{epic}

\numberwithin{equation}{section}
\newtheorem{theorem}{Theorem}[section]
\newtheorem{corollary}[theorem]{Corollary}
\newtheorem{definition}[theorem]{Definition}
\newtheorem{example}[theorem]{Example}
\newtheorem{lemma}[theorem]{Lemma}
\newtheorem{proposition}[theorem]{Proposition}

\begin{document}

\title[Branching Multiplicity Spaces]{Tiling Branching Multiplicity Spaces\\ with $GL_2$ Pattern Blocks}

\author{Sangjib Kim}

\address{School of Mathematics and Physics\\
The University of Queensland\\ 
St Lucia QLD 4072, Australia}

\email{skim@maths.uq.edu.au}

\subjclass[2010]{20G05, 05E10}

\keywords{classical groups, representations, branching rules}

\thanks{This work was supported by the UQ NSRSF}

\begin{abstract}
We study branching multiplicity spaces of complex classical groups
in terms of $GL_2$ representations. 
In particular, we show how combinatorics of $GL_2$ representations 
are intertwined to make branching rules under the restriction of 
$GL_n$ to $GL_{n-2}$. We also discuss analogous results for 
the symplectic and orthogonal groups.
\end{abstract}

\maketitle


\section{Introduction}


\subsection{}

Branching rules describe a way of decomposing an irreducible representation of 
a whole group into irreducible representations of a subgroup. With applications 
in physics, branching rules for classical groups have been extensively studied. 
See, for example, \cite{Ki75, KT90, Pr94, Wh65}.

In this paper, we study combinatorial aspects of branching rules for complex classical 
groups, under the restriction of $GL_n$ to $GL_{n-2}$, $Sp_{2n}$ to $Sp_{2n-2}$, 
and $SO_m$ to $SO_{m-2}$, by investigating the $GL_2$ module structure of branching 
multiplicity spaces. Recently, Wallach, Yacobi and the current author studied 
$Sp_{2n}$ to $Sp_{2n-2}$ branching rules in terms of $SL_2$ representations 
\cite{KY11, WY09, Ya10}. Our results for the symplectic group are compatible 
with the ones in the above papers once we restrict $GL_2$ to $SL_2$.

\subsection{}

A group homomorphism $\phi_{\alpha}$ from the complex torus $(\mathbb{C}^*)^k$ 
to $\mathbb{C}^*$ defined by 
\begin{equation*}
\phi_{\alpha}(t_1, t_2, \dotsb, t_k) 
= t_1^{\alpha_1}t_2^{\alpha_2} \cdots t_k^{\alpha_k}
\end{equation*}
is called a \textit{polynomial dominant weight} of the complex general linear 
group $GL_k=GL_k(\mathbb{C})$, if it satisfies
\begin{equation*}
\alpha = \left( \alpha_1, \dotsc, \alpha_k \right) \in \mathbb{Z}^k
\hbox{\ \  and\ \  } \alpha_1 \geq \dotsb \geq \alpha_k \geq 0.
\end{equation*}
We shall identify the polynomial dominant weight $\phi_{\alpha}$ with the exponent 
$\alpha$. We can also identify $\phi_{\alpha}$ with Young diagram having $\alpha_i$ 
boxes in the $i$th row for all $i$. The sum $\alpha_1 + \dotsb + \alpha_k$ will be
denoted by $|\alpha|$. 

\smallskip

Then, by theory of highest weight, polynomial dominant weights uniquely label 
complex irreducible polynomial representations of the general linear group, 
and we will let $V_k^{\alpha}$ denote the irreducible representation of $GL_k$
labeled by Young diagram $\alpha$, or equivalently, highest weight 
$\alpha$. See, for example, \cite[\S 9]{GW09}.

\subsection{}

The irreducible representation $V_n^{\lambda}$ of $GL_n$ labeled by Young 
diagram $\lambda$ is completely reducible as a $GL_{n-2}$ representation. 
By Schur's lemma (for example, \cite[\S 1.2]{FH91}), for a pair of polynomial 
dominant weights $\lambda$ and $\mu$ of $GL_{n}$ and $GL_{n-2}$ respectively, 
the branching multiplicity of $V_{n-2}^{\mu}$ in $V_n^{\lambda}$ is equal to 
the dimension of the space 
\begin{equation}\label{multi-space}
V^{\lambda}|_{\mu} = Hom_{GL_{n-2}} \left( V_{n-2}^{\mu}, V_n^{\lambda} \right)
\end{equation}
of $GL_{n-2}$ homomorphisms, and then, as a $GL_{n-2}$ representation, 
$V_n^{\lambda}$ decomposes into isotypic components as 
\begin{equation}\label{isotypic}
V_n^{\lambda} = \bigoplus_{\mu} V_{n-2}^{\mu} \otimes Hom_{GL_{n-2}} 
\left( V_{n-2}^{\mu}, V_n^{\lambda} \right)
\end{equation}
where the summation runs over the highest weights $\mu$ of $V_{n-2}^{\mu}$ 
appearing in $V_n^{\lambda}$. In this sense, we call the space \eqref{multi-space} 
a \textit{$GL_n$ to $GL_{n-2}$ branching multiplicity space}.

\subsection{}

After a brief review on the representations of $GL_2$ in Section \ref{sec-GL2}, 
we describe the $GL_2$ module structure of $GL_n$ to $GL_{n-2}$ branching multiplicity 
spaces in Section \ref{sec-character}. We develop a combinatorial procedure of 
\textit{tiling} branching multiplicity spaces with $GL_2$ \textit{pattern blocks} 
in Section \ref{sec-tiling}. This procedure will show, in particular, how combinatorics 
of $GL_2$ representations can be intertwined to make branching rules under the restriction 
of $GL_n$ to $GL_{n-2}$. We will discuss analogous results for the branching of $Sp_{2n}$ 
to $Sp_{2n-2}$ and $SO_m$ to $SO_{m-2}$ in Section \ref{sec-other}.

\medskip


\section{Irreducible Representations of $GL_2$}\label{sec-GL2}


In this section, we review algebraic and combinatorial models for $GL_2$ 
representations.

\subsection{}\label{review-gl2}

For a polynomial dominant weight $(x,z)\in \mathbb{Z}^2$ of $GL_2$, the irreducible 
representation with highest weight $(x,z)$ can be realized as 
\begin{equation}\label{rep-space1}
V^{(x,z)}_2 = \mathbb{C} \otimes Sym^{x-z}(\mathbb{C}^2)
\end{equation}
where $g \in GL_2$ acts on the spaces $\mathbb{C}$ and $\mathbb{C}^2$ via scaling 
by the factor of $det(g)^{z}$ and matrix multiplication, respectively. Here, 
$Sym^{d}(\mathbb{C}^2)$ denotes the $d$th symmetric power of the space $\mathbb{C}^2$, 
and $det(g)$ denotes the determinant of the matrix $g\in GL_2$. 
See, for example, \cite[\S 15.5]{FH91}.

\subsection{}

The irreducible representations of $GL_k$ can be described in terms of 
\textit{Gelfand-Tsetlin patterns} \cite{GT50}. For $GL_2$, Gelfand-Tsetlin 
patterns for $V^{(x,z)}_2$ are triangular arrays of the form
\begin{equation*}
\left[
\begin{array}{ccc}
x &   & z \\
  & y &
\end{array}
\right]
\end{equation*}
with $y \in \mathbb{Z}$ and $x \geq y \geq z$, which can label weight basis 
vectors $v \in V^{(x,z)}_2$ 
\begin{equation*}
\left(
\begin{array}{cc}
t_1 & 0 \\
0 & t_2 
\end{array}
\right) \cdot v =  ({t_1}^y {t_2}^{x+z-y}) v
\end{equation*}
for all diagonal matrices $diag(t_1,t_2)$ of $GL_2$. 
See, for example, \cite[\S 8.1]{GW09} or \cite{Mo06}.
Then, the character of the $GL_2$ representation $V^{(x,z)}_2$ is
\begin{equation}\label{GL2-character}
ch_{(x,z)} (t_1,t_2) = \sum_y {t_1}^y {t_2}^{x+z-y}
\end{equation}
where the summation runs over all integers $y$ such that $x \geq y \geq z$,
or equivalently, over all Gelfand-Tsetlin patterns with top row $(x,z)$.

\subsection{}\label{SL2GL2}

We remark that if we restrict $GL_2$ down to its subgroup $SL_2$, then 
$V^{(x,z)}_2$ is isomorphic to $Sym^{x-z}(\mathbb{C}^2)$. Its character 
can be given as, by taking $t_1=t$ and $t_2 = t^{-1}$ in (\ref{GL2-character}), 
\begin{equation*}
ch_{(d)}(t) = t^{-d} + t^{-d+2} + \dotsb + t^{d-2} + t^{d}
\end{equation*} 
where $d=x-z$. See, for example, \cite[\S 11.1]{FH91} or \cite[\S 2.3]{GW09}.

\medskip


\section{Branching Multiplicity Spaces}\label{sec-character}


In this section, we study the $GL_2$ module structure of $GL_n$ to $GL_{n-2}$
branching multiplicity spaces.

\subsection{}

Let us recall branching rules for $GL_k$ down to $GL_{k-1}$, under the embedding 
of $GL_{k-1}$ in the upper left corner of $GL_k$. For polynomial dominant weights 
$\alpha$ and $\beta$ of $GL_k$ and $GL_{k-1}$, respectively, we write 
$\beta \sqsubseteq \alpha$ and say \textit{$\beta$ interlaces $\alpha$}, if 
\begin{equation*}
\alpha_1 \geq \beta_1 \geq \alpha_2 \geq \beta_2 \geq 
\dotsb \geq \alpha_{k-1} \geq \beta_{k-1} \geq \alpha_k.
\end{equation*}

\begin{lemma}[\protect{\cite[\S 8.1]{GW09}, \cite{Mo06}}]
Let $\alpha$ and $\beta$ be polynomial dominant weights of 
$GL_k$ and $GL_{k-1}$, respectively. 
\begin{enumerate}
\item The multiplicity of a $GL_{k-1}$ irreducible representation $V_{k-1}^{\beta}$ 
in $V_{k}^{\alpha}$, as a $GL_{k-1}$ representation, is at most one. It is precisely 
one, when $\beta$ interlaces $\alpha$.
\item As a $GL_{k-1} \times GL_1$ representation, $V_k^{\alpha}$ decomposes as
\begin{equation*} 
V_k^{\alpha} = \bigoplus_{\beta \sqsubseteq  \alpha} V_{k-1}^{\beta} 
\hat{\otimes} V_1^{(|\alpha|-|\beta|)}
\end{equation*}
where the summation runs over all $\beta$ interlacing $\alpha$.
\end{enumerate}
\end{lemma}

Next, let us consider polynomial dominant weights $\lambda$ and $\mu$ of $GL_n$ 
and $GL_{n-2}$, respectively. We say \textit{$\mu$ doubly interlaces $\lambda$}, 
if there exists a polynomial dominant weight $\kappa$ of $GL_{n-1}$ such that 
$\mu$ interlaces $\kappa$ and $\kappa$ interlaces $\lambda$, i.e., 
$\mu \sqsubseteq \kappa \sqsubseteq \lambda$. By applying the above lemma twice, 
it is straightforward to see that

\begin{proposition}\label{GL-one-one}
\begin{enumerate}
\item The irreducible representation $V_{n-2}^{\mu}$ appears in 
$V_{n}^{\lambda}$ as a $GL_{n-2}$ representation
if and only if $\mu$ doubly interlaces $\lambda$.

\item The multiplicity of  $V_{n-2}^{\mu}$ in $V_{n}^{\lambda}$
is equal to the number of all possible $\kappa$'s satisfying 
$\mu \sqsubseteq \kappa \sqsubseteq \lambda$.

\item As a $GL_{n-2} \times GL_1 \times GL_1$ representation, 
$V_n^{\lambda}$ decomposes as
\begin{equation*} 
V_n^{\lambda} = \bigoplus_{\mu \sqsubseteq \kappa} \bigoplus_{\kappa \sqsubseteq \lambda} 
V_{n-2}^{\mu} \hat{\otimes} \left( V_1^{(|\kappa|-|\mu|)} \hat{\otimes} 
V_1^{(|\lambda|-|\kappa|)} \right)
\end{equation*}
where the summation runs over all $\mu$ doubly interlacing $\lambda$ and 
$\kappa$ satisfying $\mu \sqsubseteq \kappa \sqsubseteq \lambda$.
\end{enumerate}
\end{proposition}

By comparing \eqref{isotypic} and Proposition \ref{GL-one-one}, 
we can describe the branching multiplicity space 
\begin{equation*}
V^{\lambda}|_{\mu} = Hom_{GL_{n-2}} \left( V_{n-2}^{\mu}, V_n^{\lambda} \right)
\end{equation*}
in terms of integral sequences $\kappa$ such that $\mu \sqsubseteq \kappa 
\sqsubseteq \lambda$, or arrays of the form
\begin{equation*}
\left[
\begin{array}{ccccccccccc}
\lambda_1 &         &\lambda_2 &         &\lambda_3 &         & \cdots &         & \lambda_{n-1} &            & \lambda_n   \\
          & \kappa_1 &        & \kappa_2 &          &\kappa_3 &         &\cdots &              & \kappa_{n-1} &   \\
          &         &\mu_1    &         &\mu_2 &         & \cdots  &          & \mu_{n-2}    &              &                  
\end{array}
\right]
\end{equation*}
where the entries are weakly decreasing along the diagonals from left to right, 
which we will call \textit{interlacing patterns}.

\subsection{}

Our next task is to show that every $GL_n$ to $GL_{n-2}$ branching multiplicity 
space can be factored into $GL_2$ representations. For polynomial dominant weights 
$\lambda$ and $\mu$ of $GL_{n}$ and $GL_{n-2}$ respectively, let $\mathcal{IP}(\lambda,\mu)$ 
be the set of interlacing patterns whose top and bottom rows are $\lambda$ and $\mu$ 
respectively. Also, for a sequence $\sigma$ of weakly decreasing nonnegative integers 
\begin{equation*}
\sigma_1 \geq  \sigma_2 \geq  \dotsb \geq \sigma_{2n-3} \geq \sigma_{2n-2},
\end{equation*} 
let $\mathcal{GT}(\sigma)$ be the set of all $(n-1)$-tuples of Gelfand-Tsetlin 
patterns for $GL_2$ whose top rows are $(\sigma_{2i-1},\sigma_{2i})$ for $1 \leq i \leq n-1$.

\begin{theorem}\label{thm-tiling}
Let $\lambda$ and $\mu$ be polynomial dominant weights of $GL_n$ and 
$GL_{n-2}$, and $\sigma=\sigma(\lambda,\mu)$ be the sequence 
$(x_1, z_1, \dotsc, x_{n-1}, z_{n-1})$ obtained by rearranging the sequence
\begin{equation*}
(\lambda_1, \lambda_2, \dotsc, \lambda_n, \mu_1, \mu_2, \dotsc, \mu_{n-2})
\end{equation*}
in weakly decreasing order, i.e., $x_1 \geq  z_1 \geq  \dotsb \geq x_{n-1} \geq z_{n-1}$.
Then, the map from $\mathcal{IP}(\lambda,\mu)$ to $\mathcal{GT}(\sigma)$ sending
\begin{equation*}
\left[
\begin{array}{ccccccccccc}
\lambda_1 &         &\lambda_2 &         &\lambda_3 &         & \cdots &         & \lambda_{n-1} &            & \lambda_n   \\
          & \kappa_1 &        & \kappa_2 &          &\kappa_3 &         &\cdots &              & \kappa_{n-1} &   \\
          &         &\mu_1    &         &\mu_2 &         & \cdots  &          & \mu_{n-2}    &              &                  
\end{array}
\right]
\end{equation*}
to
\begin{equation*}
\left(
\begin{array}{cccc}
\left[
\begin{array}{ccc}
x_1 &   & z_1 \\
  & \kappa_1 &
\end{array}
\right], 
&
\left[
\begin{array}{ccc}
x_2 &   & z_2 \\
  & \kappa_2 &
\end{array}
\right],
&
\dots \ ,
&
\left[
\begin{array}{ccc}
x_{n-1} &   & z_{n-1} \\
  & \kappa_{n-1} &
\end{array}
\right]
\end{array}
\right)
\end{equation*}
is a bijection.
\end{theorem}

We will prove the theorem in the context of pattern-tiling in Proposition \ref{gluing}. 
Our proof will show in particular how combinatorics of $GL_2$ representations are 
intertwined to make branching rules under the restriction of $GL_n$ 
to $GL_{n-2}$. We also note that a direct proof can be given by using 
the observation that if $\mu$ doubly interlaces $\lambda$, then $x_1=\lambda_1$, $z_{n-1}=\lambda_n$, and
\begin{equation}\label{min-max}
z_j = \max(\lambda_{j+1},\mu_j) \hbox{\ \  and\ \ } x_{j+1} = \min(\lambda_{j+1},\mu_j)
\end{equation}
for $1 \leq j\leq n-2$. 


As an immediate consequence of Theorem \ref{thm-tiling}, since 
there are exactly $(x-z+1)$ possible Gelfand-Tsetlin patterns with 
top row $(x, z)$, we have

\begin{corollary}
For $\mu$ doubly interlacing $\lambda$, the multiplicity of $V^{\mu}_{n-2}$ in 
$V^{\lambda}_n$, or equivalently, the dimension of the branching multiplicity 
space $V^{\lambda}|_{\mu}$ is
\begin{equation*}
\prod_{j=1}^{n-1}(x_j - z_j +1)
\end{equation*}
where $x_j$'s and $z_j$'s are defined from the rearrangement 
$(x_1, z_1, \dotsc, x_{n-1}, z_{n-1})$ of the sequence 
$(\lambda_1, \dotsc, \lambda_n, \mu_1, \dotsc, \mu_{n-2})$ in 
weakly decreasing order.
\end{corollary}

We note that this formula can be derived from \cite[Proposition 3.2]{Ya10}. 
See the remark after Theorem \ref{thm-GL2}.

\subsection{}

In the setting of Proposition \ref{GL-one-one}, consider the diagonal 
block $GL_2$ complement to $GL_{n-2}$ in $GL_n$:
\begin{equation*}
\left[
\begin{array}{cc}
g_1 & 0     \\
 0  & g_2 
\end{array}
\right] \in GL_n
\end{equation*}
where $g_1 \in GL_{n-2}$ and $g_2 \in GL_2$. Then, this $GL_2$ commutes with $GL_{n-2}$ 
acting on $V_{n-2}^{\mu}$ in \eqref{isotypic}, and therefore, the $GL_n$ to $GL_{n-2}$ 
branching multiplicity space carries the structure of a $GL_2$ module.

\begin{theorem}\label{thm-GL2}
For $\mu$ doubly interlacing $\lambda$, the $GL_n$ to $GL_{n-2}$ branching multiplicity 
space $V^{\lambda}|_{\mu}$ is, as a $GL_2$ representation, isomorphic to the tensor product 
of $GL_2$ irreducible representations
\begin{equation*} 
Hom_{GL_{n-2}} \left( V_{n-2}^{\mu}, V_n^{\lambda} \right) \cong 
\mathbb{C} \otimes V_2^{(x_1,z_1)} \otimes V_2^{(x_2, z_2)} \otimes \dots \otimes V_{2}^{(x_{n-1}, z_{n-1})}
\end{equation*}
where $\mathbb{C}$ is the one-dimensional representation given by $det(g)^{-|\mu|}$ for $g\in GL_2$;
and $x_j$ and $z_j$ are defined from the rearrangement $(x_1, z_1, \dotsc, x_{n-1}, z_{n-1})$ 
of the sequence $(\lambda_1, \dotsc, \lambda_n, \mu_1, \dotsc, \mu_{n-2})$ in weakly decreasing order.
\end{theorem}

\begin{proof}
By taking $GL_1 \times GL_1$ in Proposition \ref{GL-one-one} as a maximal torus 
of $GL_2$, we can consider the following formula as the $GL_2$ character of the 
branching multiplicity space 
\begin{equation*}
ch(V^{\lambda}|_{\mu})=\sum_{\kappa} t_1^{|\kappa|-|\mu|} t_2^{|\lambda|-|\kappa|}
\end{equation*}
where $(t_1,t_2)\in GL_1 \times GL_1$ and the summation runs over all $\kappa$ such 
that $\mu \sqsubseteq \kappa \sqsubseteq \lambda$. Then,
\begin{eqnarray*}
&& (t_1 t_2)^{|\mu|} \cdot ch(V^{\lambda}|_{\mu}) = \sum_{\kappa}
 {t_1}^{|\kappa|} {t_2}^{|\lambda|+|\mu|-|\kappa|} \\
&=& \sum_{\kappa} {t_1}^{(\kappa_1 + \dotsb + \kappa_{n-1})}
{t_2}^{(x_1 + z_1 + \dotsb + x_{n-1} + z_{n-1}) - (\kappa_1 + \dotsb + \kappa_{n-1})} \\
&=& \prod_{j=1}^{n-1} \sum_{\kappa_j}{t_1}^{\kappa_j} {t_2}^{x_j + z_j - \kappa_j} 
\end{eqnarray*}
and, by Theorem \ref{thm-tiling}, we have $x_j \geq \kappa_j \geq z_j$ for each $j$. 
This shows that $ch(V^{\lambda}|_{\mu})$ is the product of $(t_1 t_2)^{-|\mu|}$, 
the character of the one dimensional representation twisted by $det(g)^{-|\mu|}$, 
and the characters of $V_2^{(x_j, z_j)}$'s. This finishes our proof.
\end{proof}

The following $SL_2$ module structure of the branching multiplicity space was studied 
by Yacobi in his thesis. See \cite[Proposition 3.2]{Ya10}.
\begin{equation*} 
Hom_{GL_{n-2}} \left( V_{n-2}^{\mu}, V_n^{\lambda} \right) \cong 
Sym^{x_1 - z_1}(\mathbb{C}^2) \otimes \dots \otimes Sym^{x_{n-1} - z_{n-1}}(\mathbb{C}^2).
\end{equation*}
Our theorem can be understood as a result obtained by lifting $SL_2$ to $GL_2$.

\medskip


\section{Tiling Branching Multiplicity Spaces}\label{sec-tiling} 


In this section, we develop a combinatorial procedure of tiling branching 
multiplicity spaces with Gelfand-Tsetlin patterns for $GL_2$, thereby proving 
Theorem \ref{thm-tiling}.

\subsection{}

First, in order to consider some directed paths in a graph, we place 
vertices on the coordinate plane as
\begin{equation*}\label{P-points}
P_n = \left\{(a,b): b=0, \  1 \leq a \leq n \right\} \cup 
 \left\{(a,b): b= -1, \ 2 \leq a \leq n-1 \right\}.
\end{equation*}
For example, $P_7$ is

\smallskip

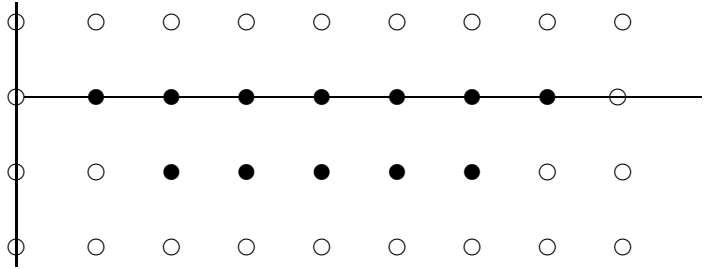
\begin{equation*}
\setlength{\unitlength}{1mm}
\begin{picture}(80, 30)
\multiputlist(0,0)(10,0){\circle{2},  \circle{2},  \circle{2}, 
\circle{2}, \circle{2}, \circle{2}, \circle{2}, \circle{2}, \circle{2}}

\multiputlist(0,10)(10,0){\circle{2}, \circle{2},  \circle*{2}, 
\circle*{2}, \circle*{2}, \circle*{2}, \circle*{2}, \circle{2}, \circle{2}}

\multiputlist(0,20)(10,0){\circle{2}, \circle*{2},  \circle*{2}, 
\circle*{2}, \circle*{2}, \circle*{2}, \circle*{2}, \circle*{2},\circle{2}}

\multiputlist(0,30)(10,0){\circle{2}, \circle{2}, \circle{2}, 
\circle{2}, \circle{2}, \circle{2}, \circle{2}, \circle{2}, \circle{2}}

\multiputlist(45,20)(10,0){{\line(1,0){90}}}

\multiputlist(-1,15)(0,10){{\line(0,1){35}}}
\end{picture}
\end{equation*}

\smallskip

Then, we consider directed paths from $u=(1, 0)$ to $v=(n, 0)$ on $(2n-3)$ 
steps visiting each point in $P_n$ exactly once, when we are only allowed 
to move right($\rightarrow$) or up($\uparrow$) or down($\downarrow$) 
or up-right($\nearrow$) or down-right($\searrow$) at each step.

\begin{example} \label{fig-path} 
These are two paths for $P_6$ out of $16$ possible ones.
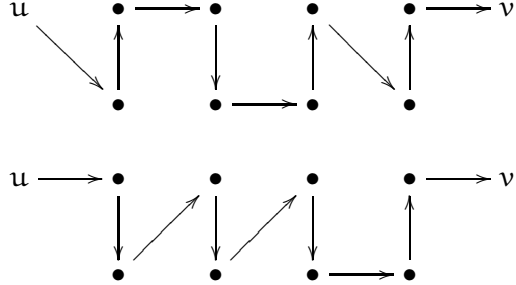
\begin{equation*}
\begin{array}{c}
\xymatrix{
u \ar[dr]  & \bullet \ar[r] & \bullet \ar[d] & \bullet \ar[dr] & \bullet \ar[r] & v \\
           & \bullet \ar[u] & \bullet \ar[r] & \bullet \ar[u]  & \bullet \ar[u]}
\\
{\ }
\\
\xymatrix{
u \ar[r]  & \bullet \ar[d] & \bullet \ar[d] & \bullet \ar[d] & \bullet \ar[r] & v \\
          & \bullet \ar[ur] & \bullet \ar[ur] & \bullet \ar[r]  & \bullet \ar[u]}              
\end{array}
\end{equation*}
\end{example}

Each directed path can be presented by a sequence of allowed steps.
For example, the two paths for $P_6$ in Example \ref{fig-path} can be 
presented as, respectively,
\begin{eqnarray*}
&& \left[ \hbox{\ } \searrow \hbox{\ \ \ } \uparrow \hbox{\ \ \ }  \rightarrow \hbox{\ \ \ }  \downarrow \hbox{\ \ \ }  
  \rightarrow  \hbox{\ \ \ } \uparrow  \hbox{\ \ \ } \searrow  \hbox{\ \ \ } \uparrow  \hbox{\ \ \ } \rightarrow \hbox{\ } \right],\\
&& \left[ \hbox{\ } \rightarrow  \hbox{\ \ \ } \downarrow  \hbox{\ \ \ } \nearrow  \hbox{\ \ \ } \downarrow  \hbox{\ \ \ } 
  \nearrow  \hbox{\ \ \ } \downarrow  \hbox{\ \ \ } \rightarrow  \hbox{\ \ \ } \uparrow  \hbox{\ \ \ } \rightarrow \hbox{\ } \right].
\end{eqnarray*}

At each step of a path, it is clear whether we are on the line $y=0$ or 
the line $y=-1$; and if we are on $y=0$ then the next step should be 
down($\downarrow$), and if we are on $y=-1$ then the next step should be 
up($\uparrow$). Therefore, in presenting directed paths for $P_n$ from 
$(1,0)$ to $(n,0)$, we may omit up($\uparrow$) and down($\downarrow$) 
arrows. Then, by denoting moving right($\rightarrow$) on the line $y=0$ 
and on the line $y=-1$ by harpoon-up($\rightharpoonup$) and 
harpoon-down($\rightharpoondown$), respectively, we can 
present every path uniquely with the following $4$ arrows:
\begin{equation*}
\searrow \hbox{\ ,\ \ \ } \rightharpoonup \hbox{\ ,\ \ \ }
\rightharpoondown \hbox{\ ,\ \ \ } \nearrow \ .
\end{equation*}

\subsection{}
 
From this observation, we define \textit{pattern blocks} attached to arrows 
and a \textit{tiling} given by a directed path.
\begin{definition}\label{arrow-table}
\begin{enumerate}
\item For each $i$ with $1 \leq i \leq n-1$, the $i$th pattern block corresponding 
to the down-right, harpoon-up, harpoon-down and up-right arrows are
\begin{center}
\begin{tabular}{|ccc|ccc|ccc|ccc|}
 \hline
 & $\searrow$ & & & $\rightharpoonup$ & & & $\rightharpoondown$ & & &$\nearrow$ & \\
 \hline \hline
$x_i$ &    &    & $x_i$   &   & $z_i$  &  &  &    &  &   & $z_i$  \\
  & $y_i$  &    &    & $y_i$  &   &   & $y_i$ &   &  & $y_i$ &    \\
  &    & $z_i$  &    &   &    & $x_i$ &   &  $z_i$  & $x_i$ &  &    \\ 
\hline
\end{tabular}
\end{center}

\item For each directed path from $(1,0)$ to $(n,0)$ of $P_n$, its tiling is 
the concatenation of pattern blocks defined by the sequence of arrows presenting 
the path such that
\begin{enumerate}
\item $y_i$ is at coordinate $(i+0.5, -0.5)$;
\item $x_i$ and $z_i$ above $y_i$ are at coordinates $(i,0)$ and $(i+1, 0)$, respectively;
\item $x_i$ and $z_i$ below $y_i$ are at coordinates $(i,-1)$ and $(i+1, -1)$, respectively
\end{enumerate}
for $1 \leq i \leq n-1$.
\end{enumerate}
\end{definition}

With this definition, the two paths given in Example \ref{fig-path} can 
be presented as
\begin{equation*}
\left[ \hbox{\ }\searrow  \hbox{\ \ \ } \rightharpoonup \hbox{\ \ \ }
\rightharpoondown \hbox{\ \ \ } \searrow \hbox{\ \ \ } \rightharpoonup \hbox{\ }\right] 
\hbox{\ \ and\ \ }
\left[ \hbox{\ } \rightharpoonup \hbox{\ \ \ } \nearrow \hbox{\ \ \ }\nearrow 
\hbox{\ \ \ } \rightharpoondown \hbox{\ \ \ } \rightharpoonup  \hbox{\ }\right],
\end{equation*}
and the corresponding tilings are
\begin{equation*}
\left[
\begin{array}{ccccccccccc}
 {x_1} &     &   {x_2} &      & {z_2} &     & {x_4} &     &  {x_5} &     &  {z_5} \\
    &  {y_1} &      & {y_2}  &     &   y_3 &     & {y_4} &      & {y_5} &     \\
    &      & {z_1}  &      &  x_3 &     &  z_3 &     &  {z_4} &     &             
\end{array}
\right]
\end{equation*} 
 and
\begin{equation*}
\left[
\begin{array}{ccccccccccc}
 {x_1} &     &   {z_1} &      & {z_2} &     &  z_3 &     &   {x_5} &     &  {z_5} \\
    &  {y_1} &      & {y_2}  &     & y_3 &     & {y_4} &      &  {y_5} &     \\
    &      & {x_2} &      &  x_3 &     & {x_4} &     &  {z_4} &     &             
\end{array}
\right]
\end{equation*}
respectively.

\subsection{}\label{lammu}

For each tiling, we identify two subsequences of $(x_1,z_1,\dotsc,x_{n-1},z_{n-1})$.
Let $\lambda=(\lambda_1,\dotsc, \lambda_n)$ be the subsequence on the line $y=0$; and
$\mu=(\mu_1, \dotsc, \mu_{n-2})$ be the be the subsequence on the line $y=-1$. 
In the above example, $\lambda$ and $\mu$ are, respectively,
\begin{eqnarray*}
&& \lambda=(x_1,x_2,z_2,x_4,x_5,z_5) \hbox{\ \ and\ \ } \mu=(z_1,x_3,z_3,z_4);\\
&& \lambda=(x_1,z_1,z_2,z_3,x_5,z_5) \hbox{\ \ and\ \ } \mu=(x_2,x_3,x_4,z_4).
\end{eqnarray*}
We note that, with the order $x_1 \geq z_1 \geq x_2 \geq z_2 \geq \dots$,
the entries of the sequences $\lambda$ and $\mu$ satisfy the identities \eqref{min-max}.

The following proposition shows that the tiling procedure given in 
Definition \ref{arrow-table} provides the correspondence stated in 
Theorem \ref{thm-tiling}.

\begin{proposition}\label{gluing}
\begin{enumerate}
\item For a given tiling, let us impose the following order on 
the entries $x_i$'s and $z_i$'s of pattern blocks
\begin{equation*} 
x_1 \geq z_1 \geq x_2 \geq z_2 \geq \dotsb \geq x_{n-1} \geq z_{n-1},
\end{equation*}
and let $\lambda$ and $\mu$ be its subsequences placed on the lines $y=0$ and $y=-1$, 
respectively. If $y_i$ satisfies $x_i \geq y_i \geq z_i$ for each pattern block, then 
$\mu \sqsubseteq (y_1, \dotsc, y_{n-1}) \sqsubseteq \lambda$, i.e., for all $r$ and $s$,
\begin{equation*}
\lambda_r \geq y_r \geq \lambda_{r+1} \hbox{\ \ \ \ and\ \ \ \ } y_s \geq \mu_s \geq y_{s+1}.
\end{equation*}

\item Conversely, let an interlacing pattern 
\begin{equation*}
\mu \sqsubseteq (y_1, \dotsc, y_{n-1}) \sqsubseteq \lambda
\end{equation*} 
be given. If we place its entries  $\lambda_i$, $\mu_j$ and $y_k$ on coordinates 
$(i,0)$, $(j+1,-1)$ and $(k+0.5,-0.5)$ for all $i$, $j$ and $k$, then we obtain a tiling 
defined by the directed path connecting $\lambda_i$'s and $\mu_j$'s in weakly decreasing order. 
That is, if $(x_1, z_1, \dotsc, x_{n-1}, z_{n-1})$ is the rearrangement of the sequence 
$(\lambda_1, \dotsc, \lambda_n, \mu_1, \dotsc, \mu_{n-2})$ in weakly decreasing order,
then $x_i$, $y_i$ and $z_i$ form a pattern block and satisfy 
\begin{equation*}
x_i \geq y_i \geq z_i
\end{equation*}
for $1 \leq i \leq n-1$.
\end{enumerate}
\begin{proof}
It is enough to check out the inequalities for all possible two consecutive pattern blocks 
in a tiling listed below. Note that these are also all possible partial interlacing patterns 
with two triples $(x,y,z)$ and $(x',y',z')$.
\begin{equation*}
\begin{array}{ccc}
\left[
\begin{array}{ccccc}
 x &     &  x' &      &    \\
   & y   &     & y'  &        \\
   &      & z  &      & z'              
\end{array} 
\right]
&
\left[
\begin{array}{ccccccc}
x &    &  x' &     & z'   \\
  & y  &     & y'  &        \\
  &    & z   &     &               
\end{array} 
\right]
&
\left[
\begin{array}{ccccccc}
  &    &  x' &     &    \\
  & y  &     & y'  &        \\
x &    & z   &     &  z'             
\end{array} 
\right]

\\

\left[
\begin{array}{ccccccc}
  &   &  x' &     & z'   \\
  & y  &     & y'  &        \\
x &    &  z &     &             
\end{array} 
\right]
&
\left[
\begin{array}{ccccccc}
x  &   &  z &      &   \\
  & y  &     & y'  &        \\
  &    &  x' &     & z'            
\end{array} 
\right]
&
\left[
\begin{array}{ccccc}
 x &     & z &      & z'   \\
   & y   &     & y'  &        \\
   &      & x'  &      &              
\end{array} 
\right]

\\

\left[
\begin{array}{ccccccc}
 &    &  z &     &    \\
  & y  &     & y'  &        \\
x &    & x'   &     & z'              
\end{array} 
\right]
&
\left[
\begin{array}{ccccccc}
  &    &  z &     & z'   \\
  & y  &     & y'  &        \\
x &    &  x'   &     &              
\end{array} 
\right]
& \ 
\end{array}
\end{equation*}

In the first case, $(\lambda_1, \lambda_2)=(x,x')$ and 
$(\mu_1,\mu_2)=(z,z')$. With $x \geq z \geq x' \geq z'$, we have 
$x \geq y \geq z$ and $x' \geq y' \geq z'$ if and only if
\begin{equation*}
x \geq y \geq x' \geq y' \hbox{\ \ and\ \ } y \geq z \geq y' \geq z'.
\end{equation*}
In the second case, $(\lambda_1, \lambda_2, \lambda_3)=(x,x',z')$ and 
$\mu_1=z$. With $x \geq z \geq x' \geq z'$, we have $x \geq y \geq z$ and 
$x' \geq y' \geq z'$ if and only if 
\begin{equation*}
x \geq y \geq x' \geq y' \geq z' \hbox{\ \ and\ \ } y \geq z \geq y'.
\end{equation*}
The rest of the cases can be shown similarly.
\end{proof}
\end{proposition}

\subsection{}\label{ex-tiling}

We give an example illustrating tiling procedures, and therefore showing 
the $GL_2$ module structure of branching multiplicity spaces.
Let us consider polynomial dominant weights 
$(x_i,z_i) \in \{(8,5),(4,2), (1,0)\}$ of $GL_2$, and Gelfand-Tsetlin patterns
\begin{equation*}
\left( \left[
\begin{array}{ccc}
8 &   & 5 \\
  & y_1 &
\end{array}
\right],
\left[
\begin{array}{ccc}
4 &   & 2 \\
  & y_2 &
\end{array}
\right],
\left[
\begin{array}{ccc}
1 &    & 0 \\
  & y_3 &
\end{array}
\right] \right)
\end{equation*}
where $y_i \in \mathbb{Z}$ varies for $x_i \geq y_i \geq z_i$ for all $i$.

\smallskip

In order to assemble these $GL_2$ pattern blocks to build $GL_4$ to $GL_2$ branching 
multiplicity spaces, we consider all the directed paths for $P_4$. 
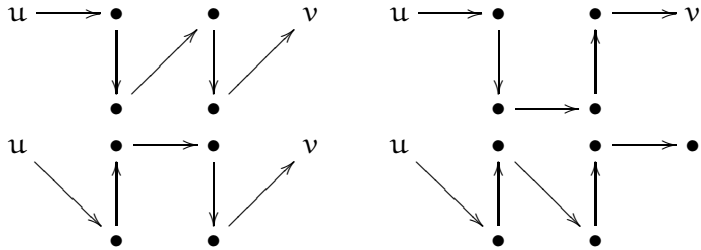
\begin{equation*}
\begin{array}{ccc}
\xymatrix{
u \ar[r]  & \bullet \ar[d] & \bullet \ar[d]  & v \\
          & \bullet \ar[ur]& \bullet \ar[ur] &  }
& &
\xymatrix{
u \ar[r]  & \bullet \ar[d] & \bullet \ar[r] & v \\
          & \bullet \ar[r] & \bullet \ar[u] &   }
\\
\xymatrix{
u \ar[dr] & \bullet \ar[r] & \bullet \ar[d]  & v \\
          & \bullet \ar[u] & \bullet \ar[ur] &  }
& &
\xymatrix{
u \ar[dr] & \bullet \ar[dr] & \bullet \ar[r] &  \bullet \\
          & \bullet \ar[u]  & \bullet \ar[u] & }               
\end{array}
\end{equation*}
They can be presented as, by using down-right, up-right, harpoon-up and 
harpoon-down arrows,
\begin{equation*}
\begin{array}{ccc}
\left[\hbox{\ \ } \rightharpoonup \hbox{\ \ \ \ } \nearrow \hbox{\ \ \ \ } \nearrow \hbox{\ \ } \right] 
& &
\left[ \hbox{\ \ }\rightharpoonup \hbox{\ \ \ \ } \rightharpoondown \hbox{\ \ \ \ } \rightharpoonup \hbox{\ \ } \right] 
\\
& &
\\
\left[\hbox{\ \ } \searrow \hbox{\ \ \ \ } \rightharpoonup \hbox{\ \ \ \ } \nearrow \hbox{\ \ } \right] 
& &
\left[\hbox{\ \ } \searrow \hbox{\ \ \ \ } \searrow \hbox{\ \ \ \ } \rightharpoonup \hbox{\ \ } \right].
\end{array}
\end{equation*}

Then, from Definition \ref{arrow-table}, we obtain the tilings
\begin{equation*}
\begin{array}{ccc}
\left[
\begin{array}{ccccccc}
 8 &     &  5 &      & 2 &     & 0   \\
   & y_1 &    & y_2  &   & y_3 &     \\
   &      & 4 &      & 1 &     &              
\end{array} 
\right]
&
\left[
\begin{array}{ccccccc}
8 &     &  5 &      & 1 &     & 0   \\
  & y_1 &    & y_2  &   & y_3 &     \\
  &      & 4 &      & 2 &     &              
\end{array} 
\right]
\\
\ & \ 
\\
\left[
\begin{array}{ccccccc}
8 &     &  4 &      & 2 &     & 0   \\
  & y_1 &    & y_2  &   & y_3 &     \\
  &      & 5 &      & 1 &     &              
\end{array} 
\right]
&
\left[
\begin{array}{ccccccc}
8 &     &  4 &      & 1 &     & 0   \\
  & y_1 &    & y_2  &   & y_3 &     \\
  &      & 5 &      & 2 &     &              
\end{array} 
\right]
\end{array}
\end{equation*}
corresponding to the branching multiplicity spaces 
\begin{equation*}
\begin{array}{ccc}
Hom_{GL_2} \left( V_2^{(4,1)}, V_4^{(8,5,2,0)} \right) & Hom_{GL_2} \left( V_2^{(4,2)}, V_4^{(8,5,1,0)} \right) \\
 & \\
Hom_{GL_2} \left( V_2^{(5,1)}, V_4^{(8,4,2,0)} \right) & Hom_{GL_2} \left( V_2^{(5,2)}, V_2^{(8,4,1,0)} \right)
\end{array}
\end{equation*}
which are, by Theorem \ref{thm-GL2}, as $GL_2$ representations, isomorphic to 
\begin{equation*}
\mathbb{C} \otimes V_2^{(8,5)} {\otimes} V_2^{(4,2)} {\otimes} V_2^{(1,0)} 
\end{equation*}
where $g\in GL_2$ acts on $\mathbb{C}$ by $det(g)^{-5}$, 
$det(g)^{-6}$, $det(g)^{-6}$ and $det(g)^{-7}$, respectively.

We note that if some of the entries in the sequence $(x_1,z_1, \dotsc, x_{n-1}, z_{n-1})$ are equal, 
then different paths may give the same tiling, and therefore the same branching multiplicity space.

\medskip


\section{Branching Multiplicity Spaces of Other Classical Groups}\label{sec-other}


As in the case of the general linear group, we can study the $GL_2$ module structure 
of branching multiplicity spaces for the symplectic group. We can also obtain similar 
results for the orthogonal group within certain stable ranges. For more about stable 
range conditions in branching rules for classical groups, we refer readers to \cite{K12}.

\subsection{}

By $Sp_{2n}$ and $SO_m$, we denote the complex symplectic group of rank $n$ 
and the complex special orthogonal group of rank $\lfloor m/2 \rfloor$, respectively. 
The dominant weights of $Sp_{2n}$ and $SO_{2n+1}$ are of the form 
$\left( \lambda_1, \lambda_2, \dotsc, \lambda_n \right) \in \mathbb{Z}^n$ 
with $\lambda_1 \geq \dotsb \geq \lambda_{n} \geq 0$; and the
dominant weights of $SO_{2n}$ are of the same form
with $\lambda_1 \geq \dotsb \geq \lambda_{n-1} \geq |\lambda_n|$. 

We will state branching rules for individual cases
(see, for example, \cite[\S 25.3]{FH91} or \cite[\S 8.1]{GW09})
with the convention of Gelfand-Tsetlin patterns, i.e. the entries 
in each array are weakly decreasing along the diagonals 
from left to right.

\subsection{}

Let $W_{2n}^{\lambda}$ be the irreducible representation of $Sp_{2n}$ with 
highest weight $\lambda$. Then for a dominant weight $\mu$ of $Sp_{2n-2}$,
the multiplicity of $W_{2n-2}^{\mu}$ in $W_{2n}^{\lambda}$ as a $Sp_{2n-2}$ representation is
equal to the number of $Sp_{2n}$ dominant weights $\kappa$ such that
\begin{equation*}
\left[
\begin{array}{ccccccccccc}
\lambda_1 &         &\lambda_2 &         & \lambda_3 &         & \cdots &        &\lambda_{n} &             \\
          &\kappa_1 &          &\kappa_2 &          &\kappa_3 &        & \cdots &           & \kappa_{n}    \\
          &         &\mu_1    &         & \mu_2  &             & \cdots &        &\mu_{n-1} &                         
\end{array}
\right]
\end{equation*}

Note that we can identify this $Sp_{2n}$ to $Sp_{2n-2}$ branching rule
with the $GL_{n+1}$ to $GL_{n-1}$ branching rule in Proposition \ref{GL-one-one}. 
Therefore, as $GL_2$ representations, we have
\begin{equation*}
Hom_{Sp_{2n-2}}\left(W_{2n-2}^{\mu},W_{2n}^{\lambda} \right) 
\cong Hom_{GL_{n-1}} \left(V_{n-1}^{\mu}, V_{n+1}^{\lambda'} \right)
\end{equation*}
where $\lambda'=(\lambda_1, \dotsc, \lambda_{n},0)$ (See \cite[Theorem 3.1]{Ya10}). 
Then, we can apply Theorem \ref{thm-tiling} to tile the $Sp_{2n}$ to $Sp_{2n-2}$ 
branching multiplicity space with $GL_2$ pattern blocks. From Theorem \ref{thm-GL2}, 
we can express the branching multiplicity space as a tensor product of $GL_2$ 
representations. Also, by restricting $GL_2$ to its subgroup $SL_2$ and using the 
explanation in Section \ref{SL2GL2}, we can obtain the $SL_2$ module structure of 
the branching multiplicity space.

We remark that Wallach and Yacobi studied $Sp_{2n}$ to $Sp_{2n-2}$ branching 
multiplicity spaces with $Sp_2 = SL_2$ and $n$ copies of $SL_2$'s in 
\cite{WY09, Ya10}, and Yacobi and the current author studied their algebraic and 
combinatorial properties in \cite{KY11}.

\subsection{}

Let $W_{2n+1}^{\lambda}$ be the irreducible representation of $SO_{2n+1}$ with
highest weight $\lambda$. Then for a dominant weight $\mu$ of $SO_{2n-1}$,
the multiplicity of $W_{2n-1}^{\mu}$ in $W_{2n+1}^{\lambda}$ as a $SO_{2n-1}$ representation is
equal to the number of dominant weights $\kappa$ of $SO_{2n}$ such that
\begin{equation*}
\left[
\begin{array}{ccccccccccc}
\lambda_1 &         &\lambda_2 &         & \lambda_3 &         & \cdots &        &\lambda_{n} &              \\
          &\kappa_1 &          &\kappa_2 &          &\kappa_3 &        & \cdots &           & |\kappa_{n}|   \\
          &         &\mu_1    &         & \mu_2  &             & \cdots &        &\mu_{n-1} &                          
\end{array}
\right]
\end{equation*}

Note that if $\mu_{n-1}=0$, then the interlacing condition makes
 $\kappa_n=0$ and this branching rule becomes
exactly the same as the $GL_{n}$ to $GL_{n-2}$ branching rule in Proposition \ref{GL-one-one}.
Therefore, if $\mu_{n-1}=0$, as $GL_2$ representations,
\begin{equation*}
Hom_{SO_{2n-1}}\left(W_{2n-1}^{\mu},W_{2n+1}^{\lambda} \right) 
\cong Hom_{GL_{n-2}} \left(V_{n-2}^{\mu'}, V_{n}^{\lambda} \right)
\end{equation*}
where $\mu'=(\mu_1, \dotsc, \mu_{n-2})$.
Similarly, if $\lambda_n =0$, as $GL_2$ representations,
\begin{equation*}
Hom_{SO_{2n-1}}\left(W_{2n-1}^{\mu},W_{2n+1}^{\lambda} \right) 
\cong Hom_{GL_{n-1}} \left(V_{n-1}^{\mu}, V_{n+1}^{\lambda'} \right)
\end{equation*}
where $\lambda'=(\lambda_1, \dotsc, \lambda_{n-1},0,0)$. 
Then, we can apply Theorem \ref{thm-tiling} and Theorem \ref{thm-GL2} 
to tile the $SO_{2n+1}$ to $SO_{2n-1}$ branching 
multiplicity space with $GL_2$ pattern blocks and to factor it into 
$GL_2$ representations or $SL_2$ representations.

\subsection{}

Let $W_{2n}^{\lambda}$ be the irreducible representation of $SO_{2n}$ 
with highest weight $\lambda$. Then for a dominant weight $\mu$ of $SO_{2n-2}$, 
the multiplicity of $W_{2n-2}^{\mu}$ in $W_{2n}^{\lambda}$ as a $SO_{2n-2}$ representation is
equal to the number of $SO_{2n-1}$ dominant weights $\kappa$ such that
\begin{equation*}
\left[
\begin{array}{ccccccccccc}
\lambda_1 &         &\lambda_2 &         & \lambda_3 &         & \cdots &        &\lambda_{n-1} &            & |\lambda_n | \\
          &\kappa_1 &          &\kappa_2 &          &\cdots &        & \kappa_{n-2} &           & \kappa_{n-1} &     \\
          &         &\mu_1    &         & \mu_2  &             & \cdots &        &\mu_{n-2} &             &  |\mu_{n-1}|              
\end{array}
\right]
\end{equation*}

If $\mu_{n-2}=0$, then the interlacing condition makes $\kappa_{n-1}=\lambda_{n}=\mu_{n-1}=0$ 
and this branching rule becomes exactly the same as the $GL_{n-1}$ to $GL_{n-3}$ branching rule
in Proposition \ref{GL-one-one}. Therefore, if $\mu_{n-2}=0$, then, as $GL_2$ representations,
\begin{equation*}
Hom_{SO_{2n-2}}\left(W_{2n-2}^{\mu},W_{2n}^{\lambda} \right) 
\cong Hom_{GL_{n-3}} \left(V_{n-3}^{\mu'}, V_{n-1}^{\lambda'} \right)
\end{equation*}
where $\mu'=(\mu_1, \dotsc, \mu_{n-3})$ and $\lambda'=(\lambda_1, \dotsc,\lambda_{n-1})$.
Similarly, if $\lambda_{n-1}=0$, then $\kappa_{n-1}= \lambda_n = \mu_{n-1}=0$ and 
as $GL_2$ representations,
\begin{equation*}
Hom_{SO_{2n-2}}\left(W_{2n-2}^{\mu},W_{2n}^{\lambda} \right) 
\cong Hom_{GL_{n-2}} \left(V_{n-2}^{\mu''}, V_{n}^{\lambda''} \right)
\end{equation*}
where $\mu''=(\mu_1, \dotsc, \mu_{n-2})$ and $\lambda''=(\lambda_1, \dotsc, \lambda_{n-2},0,0)$.
Then, we can apply Theorem \ref{thm-tiling} and Theorem \ref{thm-GL2} 
to tile the $SO_{2n}$ to $SO_{2n-2}$ branching 
multiplicity space with $GL_2$ pattern blocks and to factor it 
into $GL_2$ representations or $SL_2$ representations.

\section*{Acknowledgment}
The author thanks the late Professor Daya-Nand Verma for inspiring discussions
regarding several aspects of this work. He also thanks Piu Andamiro and Mishima 
Heihachi for their helpful suggestions and comments.

\end{document}